\documentclass{amsart}
\usepackage{amssymb}
\usepackage{mathrsfs}

\newtheorem{theorem}{Theorem}[section]
\newtheorem{lemma}[theorem]{Lemma}
\newtheorem{proposition}[theorem]{Proposition}
\newtheorem{corollary}[theorem]{Corollary}
\newtheorem{statement}[theorem]{Statement}

\theoremstyle{definition}
\newtheorem{definition}[theorem]{Definition}
\newtheorem{problem}[theorem]{Problem}
\newtheorem*{problem*}{Problem}

\theoremstyle{remark}

\newtheorem*{example}{Example}

\begin{document}

\title[Discrete Subsets in Topological Groups]{Discrete Subsets in Topological Groups\\
and Countable Extremally Disconnected Groups}

\author{Evgenii Reznichenko}
\address{Department of General Topology and Geometry, Mechanics and  Mathematics Faculty, 
M.~V.~Lomonosov Moscow State University, Leninskie Gory 1, Moscow, 199991 Russia}
\email{erezn@inbox.ru}
\thanks{This work was supported by the Russian Foundation for Basic Research (project 
no.~15-01-05369).}

\author{Ol'ga Sipacheva}
\address{Department of General Topology and Geometry, Mechanics and  Mathematics Faculty, 
M.~V.~Lomonosov Moscow State University, Leninskie Gory 1, Moscow, 199991 Russia}
\email{ovsipa@gmail.com}




\begin{abstract}
In 1967 Arhangel'skii posed the problem of the existence in ZFC of a nondiscrete extremally 
disconnected topological group. The general case is still open, but we solve Arhangel'skii's 
problem for the class of countable groups. Namely, we prove that 
the existence of a countable nondiscrete extremally disconnected group implies the existence 
of a rapid ultrafilter; hence, such a group cannot be constructed in ZFC. 
We also prove that any countable topological group 
in which the filter of neighborhoods of the identity element 
is not rapid contains a discrete set with precisely one limit point, which gives a negative 
answer to Protasov's question on the existence in ZFC of 
a countable nondiscrete group in which all discrete subsets are closed.  
\end{abstract}

\maketitle


\section*{Introduction and preliminaries}

This work was motivated by the desire to solve the following problem of 
Arhan\-gel'skii~\cite{Arhangelskii67}.

\begin{problem*}[Arhangel'skii, 1967]
Does there exist in ZFC a nondiscrete Hausdorff
extremally disconnected topological group? 
\end{problem*}

The general case is still open, but in this paper we solve Arhangel'skii's problem 
for the class of countable groups. Namely, we prove 
that \emph{the nonexistence of a countable nondiscrete Hausdorff extremally disconnected group is 
consistent with ZFC} (see Corollary~\ref{Corollary 4.3}). Since extremal disconnectedness is, 
obviously, inherited by dense subspaces, it follows that separable nondiscrete extremally 
disconnected groups cannot exist in ZFC either.

Recall that a topological space is said to be \emph{extremally disconnected} if the closure of any 
open set in this space is open (or, equivalently, the closures of any two disjoint open sets are 
disjoint). Extremal disconnectedness is a 
classical notion of topology and functional analysis, and it plays a fundamental role in Boolean 
algebra. Extremally disconnected spaces were introduced by Stone~\cite{Stone} in order to 
characterize complete Boolean algebras (a Boolean algebra is complete if and only if its Stone 
space is extremally disconnected). Gleason proved that, in the category of compact spaces, the 
extremally disconnected spaces are precisely the projective objects~\cite{Gleason}, and Strauss 
extended his result to the category of regular Hausdorff spaces and perfect maps~\cite{Strauss}. 
Moreover, each regular space $X$ is the image of a uniquely determined extremally disconnected 
space $A(X)$ under an irreducible perfect map $\pi_X$ (the pair $(A(X), \pi_X)$ is called the 
projective resolution, or absolute, of $X$); see~\cite{Woods} for details. Finally, we mention the 
classical Nachbin--Goodner--Kelley theorem, which says that the injective objects in the category 
of Banach spaces and linear contractions are the spaces of continuous functions on extremally 
disconnected compact spaces~\cite{Kelley}. 

It has long been known that an infinite  extremally disconnected topological group cannot be 
compact; moreover, it cannot contain infinite compact sets~\cite{Arhangelskii67}. However, 
Arhangel'skii's problem on the existence in ZFC of general (noncompact) extremally disconnected 
groups has not been solved so far. Still, some progress has been made. First, several consistent 
examples have been constructed~\cite{Sirota, Louveau, Malykhin75, Malykhin79, Zelenyuk96, 
Zelenyuk00}. Most of these examples are countable, although Malykhin constructed (under various 
set-theoretic assumptions) a locally uncountable separable extremally disconnected group and a 
nondiscrete extremally disconnected group in which all countable subsets are closed and 
discrete~\cite{Malykhin79}. Note that maximal topological groups (see definition in 
Section~\ref{Section 4}), which are an important special case of extremally disconnected groups, 
are always locally countable~\cite{Malykhin75, Malykhin79}. The countable version of 
Arhangel'skii's problem was posed by various authors (see, e.g., 
\cite[Problem~6]{Ponomarev-Shapiro} and \cite[Question~6.1]{Comfort-vanMill}): Does there exist a 
ZFC example of a countable nondiscrete extremally disconnected topological group? It has been 
proved that such an example cannot have maximal topology~\cite{ProtasovMS} (see also 
\cite[Corollary~5.21]{Zelenyuk11}), and it cannot contain a countable nonclosed discrete 
set~\cite{Zelenyuk06} or a sequence of countable open subgroups whose intersection has empty 
interior~\cite{Sipacheva}. 

In this paper we solve the countable version of 
Arhangel'skii's problem by proving that the existence of a countable nondiscrete extremally 
disconnected topological group implies that of a rapid filter (recall that the nonexistence 
of rapid filters is consistent with ZFC~\cite{Miller}).
Our solution is based on the following statement, which we regard as one of the two 
main results of this paper: \emph{Any countable nondiscrete Hausdorff topological group whose 
identity element has nonrapid filter of neighborhoods contains a discrete subspace with precisely 
one limit point} (Theorem~\ref{Corollary 2.4}). Thus, nondiscrete Hausdorff 
countable topological groups in which all discrete subspaces are closed cannot exist in~ZFC.  

Thanks to Malykhin's beautiful theorem that any extremally 
disconnected topological group must contain an open Boolean subgroup (i.e., a subgroup 
consisting of elements of order~2) \cite{Malykhin75}, in studying the existence 
of extremally  disconnected groups, it suffices to consider only Boolean groups.  
Our second main result is that \emph{if there are no rapid 
filters, then any countable nondiscrete Hausdorff Boolean topological group 
contains two disjoint discrete subsets for each of which the zero of the group 
is a unique limit point} (Theorem~\ref{Theorem 3.1}). 

The paper is organized as follows. In the first section we introduce and study vast sets in 
groups, which are our main technical tool. In the second section we use them to construct nonclosed 
discrete sets in countable topological groups. The third section is devoted to nonclosed discrete 
sets in countable Boolean topological groups. In the last section we collect corollaries of the 
technical results of the first three sections, answer some known questions, and ask new 
questions.

A key role in our study is played by rapid filters on $\omega$. They were introduced 
in~\cite{Mokobodzki} as filters whose elements form dominating families in ${}^\omega\omega$: 
a filter $\mathscr F$ on $\omega$ is said to be \emph{rapid} if every function 
$\omega\to\omega$ is majorized by the increasing enumeration of some element of $\mathscr F$. 
Clearly, any filter containing a rapid filter is rapid as well; thus, the existence of rapid 
filters is equivalent to that of rapid ultrafilters. Rapid ultrafilters are also known as 
semi-$Q$-point, or weak $Q$-point ultrafilters. In~\cite{Miller} Miller proved that 
the nonexistence of rapid (ultra)filters is consistent with ZFC and gave equivalent 
characterizations of rapid (ultra)filters; one of them, which is 
particularly convenient for our purposes, can be reformulated as follows: 
\emph{A filter $\mathscr F$ on $\omega$ 
is nonrapid if and only if, given any function $f\colon \omega\to\omega$, 
there exists a sequence $(T_n)_{n\in\omega}$ 
of finite subsets of $\omega$ such that each $F\in\mathscr F$ satisfies the condition 
$|F\cap T_n|\geq f(n)$ 
for some $n\in\omega$} (see \cite[Theorem~3\,(3)]{Miller}).

We also mention $Q$-point, $P$-point, and selective ultrafilters on $\omega$. An ultrafilter 
$\mathscr U$ on $\omega$ is a \emph{$P$-point}, or \emph{weakly selective}, ultrafilter if, given 
any partition $\{A_n: n\in \omega\}$ of $\omega$ (or, equivalently, any increasing sequence  
$(A_n)_{n\in \omega}$ of subsets of $\omega$) with $A_n\notin \mathscr U$, $n\in \omega$, there 
exists an $A\in \mathcal {U}$ such that $|A\cap A_n|<\aleph_0$ for all $n$. An ultrafilter 
$\mathscr U$ on $\omega$ is said to be \emph{$Q$-point}, or \emph{rare}, if, given any partition 
$\{A_n: n\in \omega\}$ of $\omega$ into finite sets, there exists an $A\in \mathcal {U}$ such that 
$|A\cap A_n|=1$ for all $n$.  An ultrafilter which is simultaneously $P$-point and $Q$-point is 
said to be \emph{selective}, or \emph{Ramsey}. Any $Q$-point ultrafilter is rapid, but not vice 
versa (see, e.g.,~\cite{Miller}). As mentioned above, the nonexistence of rapid (and, therefore, 
$Q$-point) ultrafilters is consistent with ZFC. The nonexistence of $P$-point ultrafilters is 
consistent as well (see \cite{Shelah}; Shelah's original proof is presented 
in~\cite{Wimmers(Shelah)}). However, it is still unknown whether the nonexistence of both rapid and 
$P$-point ultrafilters is consistent with ZFC. 

Given a set $X$, we use $\operatorname{Ult}(X)$ to denote the set of ultrafilters on $X$ and 
$\operatorname{Ult}^*(X)$, the set of free ultrafilters on $X$. For a topological space $X$ and a 
point $x\in X$, by $\operatorname{Ult}_x(X)$ we denote the set of ultrafilters on $X$ converging to 
$x$ (i.e., containing all neighborhoods of $x$) and by $\operatorname{Ult}^*_x(X)$, the set of free 
ultrafilters on $X$ converging to $x$. There is a natural topology on $\operatorname{Ult}(X)$, 
which turns this set into a compact extremally disconnected space, called the \emph{ultrafilter 
space} of $X$ (see, e.g., \cite{Bourbaki}); the set $\operatorname{Ult}^*(X)$, as well as 
$\operatorname{Ult}_x(X)$ and $\operatorname{Ult}^*_x(X)$ for any $x\in X$, is closed in 
$\operatorname{Ult}(X)$. If sets $X$ and $Y$ differ by finitely many elements, then 
$\operatorname{Ult}^*(X)$ coincides with $\operatorname{Ult}^*(Y)$. Each map $f\colon X\to Y$ 
induces the map $\operatorname{Ult}(f)\colon \operatorname{Ult}(X)\to \operatorname{Ult}(Y)$ 
defined by setting $\operatorname{Ult}(f)(\mathscr U) = \mathscr V$ if $f^{-1}(M)\in \mathscr U$ 
for each $M\in \mathscr V$.

Given $a<b<\omega$, we set $[a,b]=\{n\in\omega: a\leq n\leq b\}$.

For simplicity, we assume all groups considered in this paper to be infinite and all topological 
groups, infinite and Hausdorff.

\section{Vast sets}

In this section we introduce vast sets in groups and describe their properties most important for 
our purposes. 

Given a  group $G$ with identity element $e$ and a positive integer $m$, let $\Phi_m(G)$ 
denote the family of all sets $M\subset G$ satisfying the following condition: 
\begin{itemize}
\item[$(\Phi_m)$]
for any $P\in [G]^m$, there exists a $Q\in [P]^2$ such that $Q^{-1}Q\subset M$;
\end{itemize}
this condition implies, in particular, that $e\in M$. 
We set $\Phi(G)=\bigcup_m \Phi_m(G)$.

\begin{definition}
\label{Definition 1} 
Let $G$ be a group with identity element $e$. 
We say that a set $M\subset G$ is \emph{vast}
if $M\in \Phi(G)$.
Given a vast set $M$, we denote the minimum $m$ for which $M\in \Phi_m(G)$ by $J^G_M$
or simply $J_M$, 
when it is clear from the context which group $G$ is meant. 
\end{definition}

First, we note that the intersections of vast sets with $P^{-1}P$ for large 
$P$ are large. 

\begin{proposition} 
\label{Proposition 1.2} 
Suppose that $G$ is a  group, $M\in \Phi(G)$, and $n$ is a positive integer.
Then there exists a positive integer $m$ such that, 
for any $P\in [G]^m$, there is a $Q\in [P]^n$ for which 
$Q^{-1}Q\subset M$.
\end{proposition}

\begin{proof}
Let $N = \max\{J_M,  n\}$.
By virtue of Ramsey's theorem~\cite{Ramsey}, 
there exists a positive integer $m$ such that any 
2-edge-colored complete graph on $m$ vertices contains
a monochromatic clique on $N$ vertices.
Take $P\subset G$ with $|P|\geq m$. 
We set $P_0=\{\{a,b\}\in [G]^2: a^{-1}b\in M\cap M^{-1} \}$ and $P_1=[P]^2\setminus P_0$.
There exists a $Q\subset P$ with $|Q|=N$ such that either $[Q]^2\subset P_0$ 
or $[Q]^2\subset P_1$. 
Since $N\geq J_M$, it follows that $[Q]^2\cap P_0\neq \varnothing$. Therefore,  
$[Q]^2\subset P_0$ and $Q^{-1}Q\subset M$.
\end{proof}

Note also that the notion of vast sets is symmetric. The following proposition follows directly 
from the definition.

\begin{proposition} \label{Proposition 1.3-1} 
Suppose that $G$ is a group and $M\in \Phi(G)$. 
Then 
\begin{enumerate} 
\item[{\rm(i)}]
$M\cap M^{-1}\in \Phi(G)$ and $J_M=J_{M\cap M^{-1}}$\textup;
\item[{\rm(ii)}]
if $M\subset L$, then $L\in \Phi(G)$ and $J_M\geq J_L$\textup;
\item[{\rm(iii)}]
$M^{-1}\in \Phi(G)$ and $J_M=J_{M^{-1}}$.
\end{enumerate}
\end{proposition}

Vast subsets of a group are large in a certain sense. We shall see below
that vastness is organically related to another notion of largeness in semigroups, namely, 
syndeticity. This notion originated in topological dynamics 
in the context of the additive semigroup of positive integers. 
Below we define syndetic subsets of groups, although the term 
usually refers to semigroups; see \cite{Hindman} for details. 

\begin{definition}[{see \cite[Definition 4.38]{Hindman}}]
Let $G$ be a group. 
A set $Q\subset G$ is {\it syndetic} if there exists a finite set $T\subset G$ such that $TQ=G$.
\end{definition}

For $Q\subset G$, we set 
\[
I_Q=\min \{|T|: T\subset G\mbox{ and }TQ = G\};
\]
$Q$ is  syndetic if and only if $I_Q$ is finite.

Note that syndetic subgroups are precisely those of finite index, and totally bounded 
topological groups are precisely those in which all open sets are 
syndetic. 

All vast sets are syndetic. To be more precise, the following assertion holds. 

\begin{proposition} 
\label{Proposition 1.8} 
Suppose that  $G$ is a group with identity element $e$, $M \in \Phi(G)$, 
and $S\subset G$.
Then there exist finite sets $Q, R\subset S$ with $|Q|,|R|<J_M$ such that $S\subset QM$
and $S\subset MR$.
Moreover, $M$ is syndetic and $I_M < J_M$.
\end{proposition}

\begin{proof}
We can assume that $M=M^{-1}$.
Let $Q$ be a maximal subset of $S$ for which
$Q^{-1}Q\cap M\subset \{e\}$. 
Then $|Q|<J_M$ 
and, for any $s\in S\setminus Q$, there exists a $q\in Q$ such that $q^{-1}s\in M$ (because 
$Q$ is maximal and $M=M^{-1}$). 
Hence $S\subset QM$.
Repeating the same argument for $S^{-1}$ instead of $S$, we see that there exists an 
$R\subset S$ with $|R|<J_M$
such that $S^{-1}\subset MR^{-1}=(RM)^{-1}$. Hence $S\subset RM$.
To prove the second assertion, it suffices to take $S=G$.
\end{proof}

The converse is not true: there exist nonvast syndetic sets.

\begin{example} 
Let $G$ be a Boolean group with zero 0, and let $H\subset G$ be its infinite proper 
subgroup. Consider $M=G\setminus H$. We have $M=-M$, and $M$ is syndetic ($I_M=2$),  
but $M$ is not vast: 
 $Q-Q\cap M = \varnothing$ 
for any $Q\subset H$.
\end{example}

However, the ``quotient sets'' of syndetic sets are vast. 

\begin{proposition} 
\label{Proposition 1.5} 
If a subset $S$ of a group $G$ is syndetic,
then $S^{-1}S\in \Phi(G)$ and $J_{S^{-1}S}\le I_{S}+1$.
\end{proposition}

\begin{proof}
Let  $T\subset G$ be a  finite set for which $G=TS$ and $|T|=I_{S}$.
Take any $P\subset G$ with $|P|\geq |T|+1$. 
There exists a $t\in T$ for which $|P\cap tS|>1$.
Given any different $a,b\in P$ such that $a,b\in tS$, we 
have $b^{-1}a,a^{-1}b\in S^{-1}S$ and $Q^{-1}Q\subset S^{-1}S$  for $Q=\{a,b\}\in[P]^2$.
\end{proof}

This proposition implies the following two assertions. 

\begin{corollary} 
\label{Proposition 1.6} 
Any subgroup of finite index in a group~$G$ is vast in~$G$.
\end{corollary}

\begin{corollary} 
\label{Proposition 1.7} 
Any neighborhood $U$ of the identity element in a totally bounded topological group 
$G$ is vast in~$G$.
\end{corollary}

\begin{proof}
Let $V$ be a neighborhood of the identity for which $V^{-1}V\subset U$. 
Since $V$ is syndetic, it follows by Proposition~\ref{Proposition 1.5} that $U$ is vast.
\end{proof}

There are vast sets different from those provided by 
Proposition~\ref{Proposition 1.5} and Corollaries~\ref{Proposition 1.6} and 
\ref{Proposition 1.7}. A whole lot of them can be 
obtained by using the following proposition.

\begin{proposition} 
\label{Proposition 1.4} 
Let $G$ be a group. If $W\subset G$ and $W\cap W^{-1}W=\varnothing$, then $G\setminus W\in \Phi(G)$
and $J_{G\setminus W}\leq 4$.
\end{proposition}

\begin{proof}
We set $M=G\setminus W$. 
Take $P\subset G$ with $|P|= 4$; suppose that $P=\{p_0,p_1,p_2,p_3\}$.
Let us show that $P^{-1}P\cap (M\cap M^{-1})\not \subset \{e\}$. Assume that, on the contrary, 
$P^{-1}P\subset (G\setminus (M\cap M^{-1}))\cup \{e\}=
W\cup W^{-1}\cup \{e\}$. Fix any $i\le 4$. For each $j\ne i$, $j\le 4$, we have either 
$p_i^{-1}p_j\in W$ or $(p_i^{-1}p_j)^{-1}=p_j^{-1}p_i\in W$. Hence the numbers 
$s_i=|\{j: p_i^{-1} p_j\in W\}|$ and $m_i=|\{j: p_j^{-1} p_i\in W\}|$ satisfy the condition 
$s_i+m_i\geq 3$. Clearly, $\sum_{i\le 4} s_i=\sum_{i\le 4} m_i$.
Therefore, $m_n\geq 2$ for some $n$. Let  $i$ and $j$ be different numbers for which 
$g=p_i^{-1} p_n\in W$ and $h=p_j^{-1} p_n\in W$.
Then either $g^{-1}h\in W$ or $h^{-1}g\in W$. 
This contradicts the assumption $W\cap W^{-1}W=\varnothing$.
Hence $P^{-1}P\cap (M\cap M^{-1})\neq \{e\}$, i.e., there exist $a, b\in P$ such that $a\ne b$ 
and  $a^{-1}b\in M\cap M^{-1}$. Clearly, for $Q=\{a,b\}\in [P]^2$, we have 
$Q^{-1}Q\subset M$.
\end{proof}

Unlike syndetic sets, vast sets in a group form a filter by virtue of the following proposition.  

\begin{proposition} 
\label{Proposition 1.3-2} 
Suppose that $G$ is a group and 
$M_1,M_2\in \Phi(G)$. Then $M_1\cap M_2 \in \Phi(G)$.
\end{proposition}

\begin{proof}
We can assume 
without loss of generality that $M_1^{-1}=M_1$ and $M_2^{-1}=M_2$. 
Proposition~\ref{Proposition 1.2} implies the existence of a positive integer 
$m$ such that, for any $P\subset G$ with $|P|\geq m$, there exists an 
$R\subset P$ with $|R|=J_{M_1}$ for which $R^{-1}R\subset M_2$. Since $|R|\geq 
J_{M_1}$, it follows that $Q^{-1}Q\subset  M_1$ for some $Q\in [R]^2$. Hence $Q^{-1}Q\subset 
M_1\cap M_2$. 
\end{proof}

Propositions~\ref{Proposition 1.3-1} and \ref{Proposition 1.3-2}, 
together with the characterization of a nonrapid filter given in the introduction, 
imply the following technical statement, which is our main tool for constructing 
nonclosed discrete sets in groups.

\begin{statement}
\label{Theorem 1.1}
Suppose that $G$ is a 
countable 
group with identity element $e$, 
$X$ is a set, $f\colon G\to X$ is a finite-to-one map, $f(G)=X$,
$\mathscr F$ is a free filter on $G$, and $\mathscr G=\{f(F): F\in\mathscr F\}$ is a nonrapid 
free filter on $X$.
Let $(M_n)_{n\in\omega}$ be a sequence of vast subsets of $G$.
Then there exists a sequence $\xi=(x_n)_{n\in\omega}\subset G\setminus\{e\}$ such that 
\begin{enumerate}
\item[{\rm(i)}]
$\xi \setminus M_n$ is finite for each $n\in\omega$\textup;
\item[{\rm(ii)}] 
each $F\in\mathscr F$ contains $g$ and $h$ such that $f(g)\neq f(h)$ and 
$g^{-1}h\in \xi$.
\end{enumerate}
\end{statement}

\begin{proof}
In view of  Propositions~\ref{Proposition 1.3-1} and \ref{Proposition 1.3-2}, 
we can assume without loss of generality 
that $M_{n+1}\subset M_n$ and $M_n=M_n^{-1}$ for all $n\in\omega$. Since the filter $\mathscr G$ is 
nonrapid, there exists a sequence $(T_n)_{n\in\omega}$ of finite subsets of $X$ such that, given 
any $F\in\mathscr F$, we have $|f(F)\cap T_n|\geq J_{M_n}$ for some $n\in\omega$. We set 
$$
S_n=\{g^{-1}h: g,h\in f^{-1}(T_n),\ f(g)\neq f(h),\ g^{-1}h\in M_n\}
$$
and $\xi=\bigcup_n S_n$.

Let us check that (i) holds. Since the sets $S_k$ are finite, $S_k\subset M_k$, and 
$M_{k+1}\subset M_k$ for all $k\in\omega$, it follows that $\xi \setminus M_n\subset 
\bigcup_{k<n}S_k$ is finite for each~$n$.

Let us verify (ii). Take $F\in\mathscr F$. We have $|f(F)\cap T_n|\geq J_{M_n}$ for 
some $n\in\omega$. Choose $P\subset F$ so that $f(P)\subset T_n$, $|P|\geq J_{M_n}$, 
and $f(g)\neq f(h)$ for any different $g,h\in P$. There exists a $Q=\{g,h\}\in [P]^2$ 
such that $Q^{-1}Q\subset M_n$.
We have $g,h\in F$, $f(g)\neq f(h)$, $g,h\in f^{-1}(T_n)$, and $g^{-1}h\in M_n$. 
Therefore,  $g^{-1}h\in S_n \subset \xi$.
\end{proof}

\section{Discrete sequences in topological groups}

In the context of topological groups, Statement~\ref{Theorem 1.1} can be refined as follows.

\begin{statement}\label{Theorem 2.1}
Let $G$ be a countable topological group with identity element $e$. 
Suppose that $X$ is a set, $f\colon G\to X$ is a finite-to-one map, $f(G)=X$,
$\mathscr F$ is a free filter on $G$ converging to $e$, and 
$\mathscr G=\{f(F): F\in\mathscr F\}$ 
is a nonrapid free filter on $X$.
Suppose also that $(U_n)_{n\in\omega}$ is a decreasing sequence of neighborhoods of $e$ such that 
$U_n=U_n^{-1}$,  $U_{n+1}^3\subset U_n$, and   $\bigcap_n U_n=\{e\}$. 
Finally, let $(H_n)_{n\in\omega}$ be a sequence of subgroups of finite index in $G$.
Then there exists a sequence $\xi=(x_n)_{n\in\omega}\subset G\setminus \{e\}$ such that 
\begin{enumerate}
\item[{\rm(i)}]
  $\xi$ is discrete and $e$ is its only limit point;
\item[{\rm(ii)}]
  each $F\in\mathscr F$ contains $g$ and $h$ such that $f(g)\neq f(h)$ and 
$g^{-1}h\in \xi$\textup;
\item[{\rm(iii)}]
  $\xi\cap gU_{n+1}$ is finite for any $n\in\omega$ and any $g\in G\setminus U_n$\textup;
\item[{\rm(iv)}]
  $\xi\setminus H_n$ is finite for each $n\in\omega$.
\end{enumerate}
If, in addition, $U_n$ is  syndetic for each $n\in\omega$, then 
\begin{enumerate}
\item[{\rm(v)}]
  $\xi\setminus U_n$ is finite for each $n\in\omega$.
\end{enumerate}
\end{statement}

\begin{proof}
Consider $\gamma=\{gU_{n+1}: n\in\omega,\ g\in G\setminus U_n\}$.
Let us enumerate the elements of $\gamma$: $\gamma=\{W_n\subset G:n\in\omega\}$. 
Suppose that $W_n=gU_{k+1}$ 
for some $k\in\omega$ and $g\in G\setminus U_k$.
Then $W_n\cap W_n^{-1}W_n=gU_{k+1}\cap U_{k+1}^{-1}U_{k+1}=\varnothing$, because 
$U_{k+1}=U_{k+1}^{-1}$ and $g\notin U_{k+1}^{3}\subset U_k$. 
Therefore, by 
Proposition~\ref{Proposition 1.4}, all sets $W_n$ are vast, and by 
Proposition~\ref{Proposition 1.3-2} and Corollary~\ref{Proposition 1.6}, all intersections 
$W_n\cap H_n$ are vast as well.  
Statement~\ref{Theorem 1.1} implies the existence of a sequence $\xi=(x_n)_{n\in\omega}\subset 
G\setminus \{e\}$ 
satisfying conditions (ii), (iii), and (iv);  (i) follows from (ii) and (iii).

Let us check (v). Take $n\in\omega$. 
Since $U_{n+2}^{-1}U_{n+2}\subset U_{n+1}$ and the set $U_{n+2}$ is syndetic, 
it follows from Proposition~\ref{Proposition 1.5} that $U_{n+1}$ is vast.
Proposition~\ref{Proposition 1.8} implies the existence of a finite set $Q\subset G\setminus U_n$ 
for which $G\setminus U_n \subset QU_{n+1}$. 
According to (iii), $\xi\cap qU_{n+1}$ is finite for each 
$q\in Q$. Therefore, $\xi\setminus U_n$ is finite.
\end{proof}

Note that in this statement, as well as in Theorem~\ref{Theorem 2.2} and 
Corollary~\ref{Corollary 2.3} below, the subgroups 
$H_n$ are not required to be proper or different. 

Any countable topological group contains a sequence $(U_n)_{n\in\omega}$ of neighborhoods of the 
identity element satisfying the assumptions of Statement~\ref{Theorem 2.1}. Thus, 
Statement~\ref{Theorem 2.1} has the following corollary. 

\begin{corollary}\label{Corollary 2.1}
Suppose that $G$ is a countable topological group with identity element $e$, 
$X$ is a set, $f\colon G\to X$ is a finite-to-one map, $f(G)=X$,
$\mathscr F$ is a free filter on $G$ converging to $e$, 
and $\mathscr G=\{f(F): F\in\mathscr F\}$ is a nonrapid free filter on $X$.
Then there exists a sequence $\xi=(x_n)_{n\in\omega}\subset G\setminus \{e\}$ such that 
\begin{enumerate}
\item[{\rm(i)}] 
  $\xi$ is discrete, and $e$ is its only limit point;
\item[{\rm(ii)}]
  each $F\in\mathscr F$ contains $g$ and $h$ such that $f(g)\neq f(h)$ and $g^{-1}h\in \xi$.
\end{enumerate}
\end{corollary}

Corollary~\ref{Corollary 2.1} implies the following 
technical assertion needed in what follows.

\begin{statement}\label{Corollary 2.2}
Suppose that there are no rapid ultrafilters.
Let $G$ be a countable topological group with identity element $e$. 
Suppose that $Y\subset G$, $e\in \overline Y\setminus Y$, and $\{Y_n : n\in \omega\}$ 
is a partition of $Y$ into finite subsets.
Then there exists a sequence $\xi=(x_n)_{n\in\omega}\subset G\setminus \{e\}$ such that 
\begin{enumerate}
\item[{\rm(i)}] 
  $\xi$ is discrete, and $e$ is its only limit point;
\item[{\rm(ii)}]
  $\xi\subset \bigcup_{i\neq j}Y_i^{-1} Y_j$.
\end{enumerate}
\end{statement}

\begin{proof}
Take a partition  $\mathcal X$ of $G$ such that $\{Y_n : n\in \omega\}\subset \mathcal X$ 
and $\{g\} \in \mathcal X$ 
for all $g\in G\setminus Y$.
We define $f\colon G\to \mathcal X$ to be the natural map taking each element 
$g\in G$ to the (uniquely 
determined) element $f(g)$ of $\mathcal X$ containing $g$. 
Let $\mathscr F$ be a free filter on $G$ converging to $e$ and containing $Y$.
Then,  by virtue of Corollary~\ref{Corollary 2.1}, 
there exists a sequence $\xi'=(x'_n)_{n\in\omega}\subset G\setminus \{e\}$ 
satisfying the following conditions:
\begin{enumerate}
\item[(i)]
  $\xi'$ is discrete, and $e$ is its only limit point;
\item[(ii)] 
  each $F\in\mathscr F$ contains $g$ and $h$ such that $f(g)\neq f(h)$ and $g^{-1}h\in \xi'$.
\end{enumerate}
We set $\xi=\xi'\cap Z$, where $Z=\bigcup_{i\neq j}Y_i^{-1} Y_j$.
Let us check that $e\in\overline \xi$. Take neighborhoods $U$ and $V$ of $e$ in $G$ for which 
$V^{-1}V\subset U$ and let $F=V\cap Y$; then $F\in\mathscr F$. 
There exist $g,h\in F$ for which $f(g)\neq f(h)$ and $g^{-1}h\in \xi'$, 
and there exist different $i,j\in\omega$ for which $Y_i=f(g)$ and $Y_j=f(h)$. 
We have  $g^{-1}h\in \xi'\cap Y_i^{-1} Y_j \cap V^{-1}V\subset \xi\cap U$, 
i.e., $\xi\cap U\neq\varnothing$. 
\end{proof}

Now, we can prove our first theorem, which strengthens Theorem~2.1 of \cite{Keyantuoy-Zelenyuk}.

\begin{theorem}\label{Theorem 2.2}
Let $(G,\tau)$ be a countable topological group with identity element $e$ and topology $\tau$, and 
let $\mathscr F$ be a nonrapid free filter on $G$ converging to $e$. Suppose that 
$\tau_{\mathrm{m}}\subset\tau$ is a metrizable group topology on $G$ coarser than $\tau$. Finally, 
suppose that $(H_n)_{n\in\omega}$ is a sequence of subgroups of finite  index in $G$. Then there 
exists a sequence $\xi=(x_n)_{n\in\omega}\subset G\setminus \{e\}$ such that 
\begin{enumerate}
\item[{\rm(i)}]
  $\xi$ is discrete, and $e$ is its only limit point both in $(G,\tau)$ and in 
$(G,\tau_{\mathrm{m}})$\textup;
\item[{\rm(ii)}] 
  $\xi\cap F^{-1}F\neq\varnothing$ for any $F\in\mathscr F$\textup;
\item[{\rm(iii)}]
  $\xi\setminus H_n$ is finite for each $n\in\omega$.
\end{enumerate}
If, in addition,  $(G,\tau_{\mathrm m})$ is totally bounded, then 
\begin{enumerate}
\item[{\rm(iv)}]
  $\xi$ converges to $e$ in $(G,\tau_{\mathrm m})$.
\end{enumerate}
\end{theorem}

\begin{proof}
Take a sequence $(U_n)_{n\in\omega}$  of neighborhoods of $e$ open in $(G,\tau_{\mathrm m})$ and 
such that $(U_n)_n$ is a base of neighborhoods of $e$ in $(G,\tau_{\mathrm m})$, $U_n=U_n^{-1}$, 
and $U_{n+1}^3\subset U_n$ for $n\in\omega$. Let $X=G$, and let $f\colon G\to X$ be the identity 
map. By virtue of Statement~\ref{Theorem 2.1}, 
there is a sequence $\xi=(x_n)_{n\in\omega}\subset 
G\setminus \{e\}$ satisfying conditions (i)--(iv) of 
Statement~\ref{Theorem 2.1}.
Clearly, this sequence satisfies also conditions (i), 
(ii), and (iii) of the theorem being proved. 

Let us check (iv). Take $n\in\omega$. 
By Corollary~\ref{Proposition 1.7} 
the neighborhood $U_{n+1}$ is vast. 
Proposition~\ref{Proposition 1.8} implies the existence of a finite set $Q\subset G\setminus U_n$ 
for which $G\setminus U_n \subset QU_{n+1}$. 
According to Statement~\ref{Theorem 2.1}\,(iii), 
$\xi\cap qU_{n+1}$ is finite for each 
$q\in Q$. Therefore, $\xi\setminus U_n$ is finite.
\end{proof}

Obviously, the topology of any countable topological group can be weakened to a metrizable group 
topology (see, e.g., \cite{Arhangelskii79}).
Thus, we obtain the following corollary of Theorem~\ref{Theorem 2.2}. 

\begin{corollary}\label{Corollary 2.3}
Let $(G,\tau)$ be a countable nondiscrete topological group with identity element $e$ such that 
the filter of neighborhoods of $e$ is nonrapid. 
Suppose that $(H_n)_{n\in\omega}$ is a sequence of subgroups of finite index in $G$.
Then there exists a sequence $\xi=(x_n)_{n\in\omega}\subset G\setminus \{e\}$ such that 
\begin{enumerate}
\item[{\rm(i)}]
  $\xi$ is discrete and $e$ is its only limit point;
\item[{\rm(ii)}]
  $\xi\setminus H_n$ is finite for each $n\in\omega$.
\end{enumerate}
\end{corollary}

A special case of this corollary is the following theorem, which is one of the main results of 
this paper.

\begin{theorem}\label{Corollary 2.4}
Any countable nondiscrete topological group whose identity element has nonrapid filter of 
neighborhoods contains a discrete sequence with precisely one limit point. 
\end{theorem}

The following theorem says that not only does any countable group with nonrapid neighborhood 
filter of the identity contain a discrete set with one limit point, but it must also contain two such 
disjoint sets with the same limit point under certain set-theoretic assumptions.

\begin{theorem}\label{Theorem 2.3}
Let $(G,\tau)$ be a countable nondiscrete topological group with identity element $e$ such that 
the filter of neighborhoods of $e$ is nonrapid, 
and let $(U_n)_{n\in\omega}$ be a decreasing sequence of neighborhoods of $e$ such that $U_0=G$, 
$U_n=U_n^{-1}$,  $U_{n+1}^3\subset U_n$, and   $\bigcap_n U_n=\{e\}$. 
Consider the map $\theta\colon G\setminus\{e\}\to \omega$ defined by 
$\theta^{-1}(n)=U_{n}\setminus U_{n+1}$ for each $n\in\omega$.
Suppose that there exist no two disjoint discrete sequences $\xi,\xi'\subset G\setminus \{e\}$
each of which has the unique limit point $e$. 
Then 
\begin{enumerate}
\item[{\rm(i)}]
$\operatorname{Ult}(\theta)(\operatorname{Ult}^*_e(G))$ contains a $P$-point ultrafilter $\mathscr U$.
\end{enumerate}
If, in addition, $U_n$ is syndetic for each $n\in\omega$, then 
\begin{enumerate}
\item[{\rm(ii)}]
 $\mathscr U$ 
can be mapped to a selective ultrafilter.
\end{enumerate}
\end{theorem}

\begin{proof} 
We set $\mathscr F$ to be the filter of neighborhoods of $e$ and $f$ to be the identity map 
$G\to G$ and apply Statement~\ref{Theorem 2.1}. 
Let $\xi=(x_n)_{n\in\omega}\subset G\setminus \{e\}$ 
be a sequence with the properties specified in 
Statement~\ref{Theorem 2.1}.
For each $n\in\omega$, there exists a 
$k_n' \in\omega$ such that $\xi\cap x_nU_{k_n'}=\{x_n\}$, because 
$\xi\cap x_nU_{\theta(x_n)+1}$ is finite and $\bigcap_m U_m=\{e\}$.
Let $(k_n)_{n\in\omega}\subset \omega$ be an increasing sequence such that $k_n>k_n'$ and 
$k_n> \theta(x_n)$. Then (a)~the sets $x_nU_{k_n}$ are disjoint and 
(b)~$\overline{\bigcup_n x_nU_{k_n}}\setminus \bigl(\bigcup_n \overline{x_nU_{k_n}}\bigr)= \{e\}$. 
Indeed, if $x_lU_{k_l}\cap x_mU_{k_m}\ne \varnothing$ and, say, $l<m$, then 
$x_m\in x_lU_{k_l}U_{k_m}^{-1} \subset x_lU_{k_l}^2\subset x_lU_{k'_l}$, which contradicts 
the definition of $k'_l$ and, thereby, proves~(a). To prove~(b), 
we take any $g\ne e$ and find $n$ for which 
$g\notin U_n$. By condition~(iii) in 
Statement~\ref{Theorem 2.1}, 
$\xi\cap gU_{n+1}$ is finite, and hence so is the set $M$ of numbers $m$ for which 
$x_mU_{n+2}\cap gU_{n+2}\ne \varnothing$; therefore, 
the intersection $x_lU_{k_l}\cap gU_{n+2}$ can be nonempty only if $l\in M$ or $k_l<n+2$, 
and the number of such $l$'s is finite. 

Let us prove~(i).
Take an ultrafilter $\mathscr V\in \operatorname{Ult}^*_e(G)$ containing $\xi$. 
We claim that $\mathscr U=\operatorname{Ult}(\theta)(\mathscr V)$ is $P$-point.
Suppose that, on the contrary, there exists an increasing sequence $(A_n)_{n\in \omega}$ 
of sets $A_n\subset \omega$ not belonging to $\mathscr U$ and such that each 
$P\in \mathscr U$ has infinite intersection with some $A_n$.
We set $B_n=\theta^{-1}(A_n)\cap U_{k_n} \cap \xi$ for $n\in\omega$ and define 
$\xi'$ as $\bigcup_n x_n B_n $. For each $n$, 
$\xi\setminus \theta^{-1}(A_n)\in \mathscr V$ 
and hence $e\notin\overline{\theta^{-1}(A_n)\cap \xi}$: otherwise, we would have 
two disjoint discrete sequences each of which has the unique limit point $e$. Therefore, 
each $B_n$ is a closed discrete set; by virtue of assertions~(a) and~(b) at the end of the preceding 
paragraph, the whole sequence $\xi'$ is discrete and cannot have limit points different from $e$. 
Note that $\xi\cap\xi'=\varnothing$. Indeed, for each $n$, $e\notin B_n$ and hence 
$x_n\notin x_nB_n$; on the other hand, $x_nB_n\cap\xi\subset x_nU_{k_n}\cap\xi=\{x_n\}$. 
Since $e\in\overline\xi$, it follows that 
$e\notin \overline{\xi'}$, i.e., $\xi'$ is a closed discrete subset of $G$.
Let $U$ be a neighborhood of $e$ with the properties $U=U^{-1}$ and $U^2\cap\xi'=\varnothing$, 
and let $P=\theta(U\cap \xi)$. We have $P\in \mathscr U$; hence there exists an $n\in\omega$ 
for which $|P\cap A_n|=\aleph_0$. 
Thus, we can choose $l,m\in\omega$ so that $x_l, x_m\in U\cap \xi$, 
$\theta(x_l),\theta(x_m)\in A_n$, and $m>k_l$.
We have $x_m \in B_l$ and $x_l x_m\in \xi' \cap U^2$. This contradiction proves that 
$\mathscr U$ is a $P$-point ultrafilter.

To prove the second assertion of the theorem, 
we need the following lemma, which is also used in the next section.

\begin{lemma}\label{Lemma 2.1}
Let $\mathscr U$ be a free ultrafilter on $\omega$, and let $\phi\colon \omega\to\omega$ be a 
monotone function such that $\phi(n)>n$ for all $n\in\omega$.
Then there exist monotone sequences $(a_n)_{n\in\omega},(b_n)_{n\in\omega}\subset \omega$ such that 
$a_n<b_n<\phi(b_n)<a_{n+1}$ for all $n\in\omega$ and $\bigcup_n[a_n,b_n]\in \mathscr U$.
\end{lemma}

\begin{proof}
Let $(c_n)_{n\in\omega}\subset \omega$ be a sequence satisfying the conditions 
$c_0=0$ and $c_{n+1} > \phi(c_n)$.
We set $A=\bigcup_n[c_{2n},c_{2n+1}]$ and $B=\bigcup_n[c_{2n+1},c_{2n+2}]$. 
We have $A\cup B=\omega$, so that either $A\in \mathscr U$ or $B\in \mathscr U$. It remains to set 
$a_n=c_{2n}$ and $b_n=c_{2n+1}$ in the former case and $a_n=c_{2n+1}$ and $b_n=c_{2n+2}$ 
in the latter.
\end{proof}

We proceed to prove assertion~(ii). Suppose that all $U_n$ are syndetic. Let us 
show that $\mathscr U$ can be mapped to a selective ultrafilter in this case. We can assume without 
loss of generality that $\theta(x_0)=0$. Recall that the sequence $\xi$ was chosen to satisfy all 
conditions in Statement~\ref{Theorem 2.1}. 
By condition~(v), $\theta^{-1}(n)\cap \xi$ 
is finite for each $n\in\omega$. Consider the function $\phi\colon \omega\to \omega$ defined by 
$$
\phi(n)=\max\{k_m: m\in\omega, \ \theta(x_m)\leq n\}
$$
for each $n\in\omega$.

By virtue of Lemma~\ref{Lemma 2.1}, there exist monotone sequences 
$(a_n)_{n\in\omega},(b_n)_{n\in\omega}\subset \omega$ such that $a_n<b_n<\phi(b_n)<a_{n+1}$ for all 
$n\in\omega$ and $C=\bigcup_n[a_n,b_n]\in \mathscr U$. Consider the map $\eta\colon C\to \omega$ 
defined by $\eta^{-1}(n)=[a_n,b_n]$ for each $n\in\omega$. We set $\mathscr 
W=\operatorname{Ult}(\eta)(\mathscr U)$ and claim that $\mathscr W$ is a $Q$-point ultrafilter. 

Indeed, suppose that, on the contrary, $\omega$ can be partitioned into disjoint finite sets 
$A_n$, $n\in \omega$, so that, for each $R\in \mathscr W$, there exists an $n\in\omega$  
such that $|R\cap A_n|>1$.
Let $D=\{n\in\omega: \theta(x_n)\in C\}$. Then the sequence 
$\xi_D=(x_n)_{n\in D}$ accumulates at $e$, because 
$\xi_D=\theta^{-1}(C)\cap\xi\in \mathscr V$.
For each $n\in D$, we find $\alpha_n\in\omega$ for which $\eta(\theta(x_n))\in A_{\alpha_n}$ and set 
$$
B_n=\{x_m\in \xi_D: \theta(x_m)\geq k_n, \ \eta(\theta(x_m)) \in A_{\alpha_n}\}.
$$
Let $\xi'=\bigcup_{n\in D}x_nB_n $. Note that each $B_n$ is finite (because $A_{\alpha_n}$ is 
finite, the map $\eta$ is finite-to-one by definition, and $\theta\restriction \xi$ is 
finite-to-one by condition~(v) in Statement~\ref{Theorem 2.1}), and 
$B_n\subset U_{k_n}$ (by the definition of the map $\theta$). Thus, for the same reasons as 
in the proof of assertion~(i), $\xi'$ is a discrete sequence having no limit points in 
$G\setminus \{e\}$, and $\xi'\cap \xi=\varnothing$. By the assumption concerning disjoint sequences 
with limit point $e$, we have $e\notin \overline{\xi'}$. 
Let $U$ be a neighborhood of $e$ such that $U=U^{-1}$ and $U^2\cap\xi'=\varnothing$.
Consider $P=\theta(U\cap \xi_D)$ and  $R=\eta(P)$. 
We have $P\in \mathscr U$; therefore, $R\in \mathscr W$. By assumption we can find 
$n\in\omega$ for which $|R\cap A_n|>1$.
Take $r,s\in R\cap A_n$, $r<s$. We have $r=\eta(\theta(x_l))$ and $s=\eta(\theta(x_m))$ 
for some different $x_l,x_m\in U\cap \xi_D$. This means that $\theta(x_l)\in [a_r, b_r]$ 
and $\theta(x_m)\in [a_s, b_s]$. By the definition of the sequences $(a_n)$ and $(b_n)$, 
we have $\theta(x_m) > \phi(b_r)$. On the other hand, since $\theta(x_l)\le b_r$, it follows that 
$\phi(b_r)\ge k_l$. Therefore, $\theta(x_m)\ge k_l$. Finally, we have $\alpha_l=n$, because 
$\eta(\theta(x_l))\in A_n$. Thus, $x_m \in B_l$, whence $x_l x_m \in U^2\cap \xi'$. 
This contradiction proves 
that the ultrafilter $\mathscr W$ is  $Q$-point. 

To complete the proof of the theorem, it remains to note that the property of being 
$P$-point is, obviously, preserved by maps of ultrafilters and that the selective ultrafilters are 
precisely those which are simultaneously $P$-points and $Q$-points. 
\end{proof}

\begin{corollary}\label{Corollary 2.5}
Let $(G,\tau)$ be a countable nondiscrete extremally disconnected topological group 
with identity element $e$ such that 
the filter of neighborhoods of $e$ is nonrapid. 
Suppose that $(U_n)_{n\in\omega}$ is a decreasing sequence of clopen neighborhoods of $e$ such that 
$U_n=U_n^{-1}$,  $U_{n+1}^3\subset U_n$, and   $\bigcap_n U_n=\{e\}$. 
Then the family 
$$
\mathscr U =\{\{n: V\cap U_n\setminus U_{n+1}\neq\varnothing\}: V\text{ is a neighborhood of }e\}
$$
is a $P$-point ultrafilter on $\omega$. 
If, moreover, all sets $U_n$ are syndetic, then $\mathscr U$ can be mapped to a selective 
ultrafilter. 
\end{corollary}

Indeed, given any set $S\subset \omega$, we have either $S\notin \mathscr U$ or $\omega\setminus 
S\notin \mathscr U$ by virtue of extremal disconnectedness. Thus, $\mathscr U$ is an ultrafilter, 
and $\{\mathscr U\}= \operatorname{Ult}(\theta)(\operatorname{Ult}^*_e(G))$. It remains to apply 
Theorem~\ref{Theorem 2.3}.

\section{Discrete sequences in Boolean groups}

All countable Boolean groups are isomorphic to each other and to the group 
$[\omega]^{<\omega}$ of finite subsets of $\omega$ with the operation $\triangle$ 
of symmetric difference 
defined by $A\triangle B=(A\setminus B)\cup (B\setminus A)$ for $A, B\in [\omega]^{<\omega}$; 
the zero of $[\omega]^{<\omega}$ is the empty set $\varnothing$. 
We also use the additive notation: $A+B=A \triangle B$ and $\boldsymbol 0=\varnothing$. 
Given a nonempty set $A\in[\omega]^{<\omega}$, 
by $\min A$ and $\max A$ we denote the minimum and maximum elements of $A$ as a subset of $\omega$. 

In this section, we identify all countable Boolean groups with $[\omega]^{<\omega}$.

The proof of our main theorem on Boolean groups is based on two lemmas.

\begin{lemma}\label{Lemma 3.1}
Suppose that $\mathscr U$ is a free ultrafilter on $[\omega]^{<\omega}$, $\xi=(X_n)_{n\in\omega}\in 
\mathscr U$, and $\lim_{n\to \infty} \min X_n=\infty$. Then there exists a sequence $(\mathcal 
Y_n)_{n\in\omega}$ of finite subsets of $\xi$ such that $\bigcup_{n\in \omega} \mathcal Y_n\in 
\mathscr U$ and $\bigl(\bigcup \mathcal Y_i\bigr) \cap \bigl(\bigcup \mathcal 
Y_j\bigr)=\varnothing$ for any different $i, j\in \omega$.
\end{lemma}

\begin{proof}
Let $\mathscr V=\min \mathscr U=\{\{\min X: X\in \mathcal M\}: \mathcal M\in \mathscr U\}$.
We assume that $\min X_0=0$.
Given $n\in\omega$, we set 
$h(n)=\max\{\max X: X\in\xi,\ \min X\leq n\}$ and $f(n)=1+\max\{h(n),n\}$.
Using Lemma~\ref{Lemma 2.1}, we choose monotone sequences 
$(a_n)_{n\in\omega},(b_n)_{n\in\omega}\subset \omega$ so that 
$a_n<b_n<f(b_n)<a_{n+1}$ for all $n\in\omega$ and $\bigcup_{n\in\omega}[a_n,b_n]\in \mathscr V$.
Let $\mathcal Y_n=\{X\in\xi: \min X\in [a_n,b_n]\}$. Then  $\bigcup_{n\in \omega} \mathcal Y_n\in 
\mathscr U$.
Since $\bigcup \mathcal Y_n\subset [a_n,a_{n+1}-1]$ for each $n\in \omega$, 
it follows that the family 
$\{\bigcup \mathcal Y_n: n\in \omega\}$ is disjoint.
\end{proof}

\begin{lemma}\label{Lemma 3.2}
Let $G=[\omega]^{<\omega}$ be a countable nondiscrete Boolean topological group in which the filter 
of neighborhoods of zero is nonrapid. Then there exists a sequence $\xi=(X_n)_{n\in\omega}\subset 
G\setminus \{\boldsymbol 0\}$ such that 
\begin{enumerate}
\item[{\rm(i)}]
  $\xi$ is discrete, and its only limit point is $\boldsymbol 0$\textup;
\item[{\rm(ii)}]
  $\xi$ can be partitioned into finite subsets $\mathcal Y_n$, $n\in \omega$, so that  
  $(\mathcal Y_i+\mathcal Y_j)\cap\xi=\varnothing$ for different $i,j\in\omega$.
\end{enumerate}
\end{lemma}

\begin{proof}
Consider the sets $H_n=[\{m\in\omega: m\geq n\}]^{<\omega}$, $n\in\omega$; these are subgroups 
of finite index in $G$. By Corollary~\ref{Corollary 2.3},
there exists a discrete sequence $\xi'={(X_n')}_{n\in\omega}\subset G\setminus \{\boldsymbol 0\}$ 
such that $\boldsymbol 0$ is its only limit point and $\xi'\setminus H_n$ is finite for each 
$n\in\omega$. We have $\lim_{n\to\infty} \min X_n'=\infty$. Let $\mathscr U$ be an ultrafilter on 
$G$ converging to $\boldsymbol 0$ and containing $\xi'$ as an element. Using Lemma~\ref{Lemma 3.1}, 
we choose a sequence $(\mathcal Y_n)_{n\in\omega}$ of finite subsets of $\xi'$ so that 
$\bigcup_{n\in \omega} \mathcal Y_n\in \mathscr U$ and $\bigl(\bigcup \mathcal Y_i\bigr) \cap 
\bigl(\bigcup \mathcal Y_j\bigr)=\varnothing$ for any different $i, j\in \omega$. It remains to set 
$\xi=\bigcup_{n\in \omega} \mathcal Y_n$. 
\end{proof}

\begin{theorem}\label{Theorem 3.1}
Suppose that there are no rapid ultrafilters.
Let $G$ be a countable nondiscrete Boolean topological group. 
Then there exist two disjoint discrete sequences 
$(X_n)_{n\in\omega},(Y_n)_{n\in\omega}\subset G\setminus \{\boldsymbol 0\}$ 
for each of which $\boldsymbol 0$ is a unique limit point. 
\end{theorem}

\begin{proof}
By Lemma~\ref{Lemma 3.2}, there is a sequence $\xi=(X_n)_{n\in\omega}\subset G\setminus 
\{\boldsymbol 0\}$ such that
\begin{enumerate}
\item[{\rm(i)}]
  $\xi$ is discrete, and $\boldsymbol 0$ is its only limit point; 
\item[{\rm(ii)}]
  $\xi$ can be partitioned into finite subsets $\mathcal Y_n$, $n\in \omega$, so that  
  $(\mathcal Y_i+\mathcal Y_j)\cap\xi=\varnothing$ for different $i,j\in\omega$.
\end{enumerate}
Using Statement~\ref{Corollary 2.2}, 
we find $\xi'=(Y_n)_n\subset G\setminus \{\boldsymbol 0\}$ 
such that 
\begin{enumerate}
\item[{\rm(iii)}]
  $\xi'$ is discrete, and $\boldsymbol 0$ is its only limit point;
\item[{\rm(iv)}] 
  $\xi'\subset \bigcup_{i\neq j}(\mathcal Y_i+\mathcal Y_j)$.
\end{enumerate}
It follows from (ii) and (iv) that $\xi\cap\xi'=\varnothing$.
\end{proof}

\section{Answers and questions}
\label{Section 4}

Theorem~\ref{Corollary 2.4} 
solves a problem of Protasov~\cite{Protasov}. 
Namely, the following assertion is valid.  

\begin{corollary} \label{Corollary 4.1}
It is consistent with ZFC that any countable nondiscrete 
topological group contains a nonclosed 
discrete subset with only one limit point. 
\end{corollary}

This assertion gives also a partial answer to Arhangel'skii  and Collins' question on the 
existence in ZFC of a nondiscrete nodec topological group~\cite[Problem~8.1]{Arhangelskii-Collins}.

According to Theorem~\ref{Theorem 2.3}, the existence of a countable nondiscrete topological group 
containing no two disjoint discrete sequences for each of which the identity is a unique limit 
point implies the existence of either a rapid ultrafilter or a $P$-point ultrafilter. As mentioned 
in the introduction, it is unknown whether the nonexistence of both rapid and $P$-point 
ultrafilters is consistent with ZFC.  This gives rise to the following question.

\begin{problem}\label{Problem 1}
Does there exist in ZFC a countable nondiscrete topological group 
containing no two disjoint discrete sequences which have the same unique limit point?  
\end{problem}

Note that such a group cannot be Boolean by virtue of Theorem~\ref{Theorem 3.1}. 

Recall that a  topological space is said to be \emph{resolvable}  
if it can be partitioned into two dense subsets; otherwise, a space is \emph{irresolvable}.
A topological space  is said to be \emph{$\omega$-resolvable} 
if it can be represented as a countable disjoint union of dense subsets. 
Any homogeneous regular space containing a countable discrete nonclosed set is $\omega$-resolvable
(see \cite[Theorem~3.33]{Zelenyuk11}).
Therefore, Theorem~\ref{Corollary 2.4} 
implies the following assertion. 

\begin{corollary} 
\label{Corollary 4.2}
The neighborhood filter of the identity element of any countable nondiscrete 
non-$\omega$-resolvable topological group is rapid.
\end{corollary}

Recall that a topological group $G$ is said to be \emph{maximal} if $G$ with any stronger (not 
necessarily group) topology has isolated points. Clearly, any maximal group is irresolvable. 
Moreover, it is known that any maximal group is locally countable and even contains 
a countable open Boolean subgroup~\cite{Malykhin79} (see also \cite[Theorem~5.7]{Zelenyuk11}). 
Therefore, Corollary~4.2 has the following consequence. 

\begin{corollary} 
\label{Corollary 4.2m}
The neighborhood filter of the identity element of any maximal topological group is rapid.
\end{corollary}

The existence of a countable nondiscrete $\omega$-irresolvable topological
group implies the existence of a $P$-point in $\beta\omega\setminus \omega$
(see \cite[Theorem~12.13]{Zelenyuk11}).

\begin{problem}\label{Problem 2}
Does the existence of a countable nondiscrete 
maximal (irresolvable, $\omega$-ir\-re\-solv\-able)  topological
group imply the existence of a selective ultrafilter?
\end{problem}

As is known, if $X$ and $Y$ are countable separated sets in an extremally disconnected 
space (``separated'' means that $\overline X\cap Y = X\cap \overline Y = \varnothing$), then 
$\overline X\cap  \overline Y = \varnothing$ (see, e.g., \cite[Proposition~1.9]{Frolik}). 
Combining this 
with Theorem~\ref{Theorem 3.1} and recalling that any extremally disconnected group contains 
an open Boolean subgroup, we arrive at the following conclusion. 

\begin{corollary}\label{Corollary 4.3}
The existence of  a countable nondiscrete extremally disconnected group implies the 
existence of a rapid ultrafilter.
\end{corollary}

Corollary~\ref{Corollary 4.3} solves Arhangel'skii's problem mentioned in the introduction for 
countable groups. 

\begin{problem}\label{Problem 3}
Is it true that the neighborhood filter of the identity element 
of any countable nondiscrete  extremally disconnected group is rapid?
\end{problem}

All examples of nondiscrete extremally disconnected groups known to the authors 
are constructed in models with selective ultrafilters. Note that 
the existence of a countable nondiscrete extremally disconnected group 
containing a nonclosed discrete subset implies that of a $P$-ultrafilter~\cite{Zelenyuk06}.

\begin{problem}\label{Problem 4}
Does the existence of a countable nondiscrete extremally disconnected group imply that of 
\begin{enumerate}
\item[{\rm(a)}] a selective ultrafilter;
\item[{\rm(b)}] a $P$-point ultrafilter;
\item[{\rm(c)}] a $Q$-point ultrafilter?
\end{enumerate}
\end{problem}

Corollary~\ref{Corollary 4.3} can be refined as follows: \emph{If $G$ is a countable nondiscrete 
extremally disconnected topological group, then some ultrafilter $\mathscr U\in 
\operatorname{Ult}_e(G)$ can be finite-to-one mapped to a rapid ultrafilter on $\omega$}. This 
suggests the following more specific formulation of Problem~\ref{Problem 4}.

\begin{problem}\label{Problem 5}
Let $G$ be a countable nondiscrete extremally disconnected topological group. 
Does there exist an ultrafilter $\mathscr U \in \operatorname{Ult}_e(G)$ that can be mapped to 
\begin{enumerate}
\item[{\rm(a)}] a selective ultrafilter;
\item[{\rm(b)}] a $P$-point ultrafilter;
\item[{\rm(c)}] a $Q$-point ultrafilter? 
\end{enumerate}
\end{problem}

\section*{Acknowledgments}

The authors are very grateful to the referee for useful comments, which helped to 
considerably improve the exposition.

\bibliographystyle{amsplain}

\providecommand{\bysame}{\leavevmode\hbox to3em{\hrulefill}\thinspace}
\providecommand{\MR}{\relax\ifhmode\unskip\space\fi MR }
\providecommand{\MRhref}[2]{%
  \href{http://www.ams.org/mathscinet-getitem?mr=#1}{#2}
}
\providecommand{\href}[2]{#2}

\end{document}